\documentclass[12pt]{amsproc}
\usepackage{eurosym}
\usepackage{amssymb}
\usepackage{amsfonts}

\theoremstyle{plain}

\newtheorem{corollary}{Corollary}

\newtheorem{lemma}{Lemma}

\newtheorem{theorem}{Theorem}
\numberwithin{equation}{section}

\begin{document}
\title{On-line list coloring of matroids}
\author{Micha\l\ Laso\'{n}}
\address{Institute of Mathematics of the Polish Academy of Sciences, 00-956 Warszawa, Poland}
\address{Theoretical Computer Science Department, Faculty of Mathematics and
Computer Science, Jagiellonian University, 30-348 Krak\'{o}w, Poland}

\email{michalason@gmail.com}
\thanks{M.~Laso\'n is supported by Polish National Science Centre under grant no N~N206 568240.}

\author{Wojciech Lubawski}
\address{Theoretical Computer Science Department, Faculty of Mathematics and
Computer Science, Jagiellonian University, 30-348 Krak\'{o}w, Poland}
\address{Institute of Mathematics of the Polish Academy of Sciences, 00-956 Warszawa, Poland}
\email{w.lubawski@gmail.com}
\thanks{W.~Lubawski is supported by joint programme SSDNM}


\begin{abstract}
A coloring of a matroid is \emph{proper} if elements of the same color form
an independent set. A theorem of Seymour asserts that a $k$-colorable
matroid is also colorable from any lists of size $k$. In this note we
generalize this theorem to the on-line setting. We prove that a coloring of a
matroid from lists of size $k$ is possible even if appearances of colors in the lists are recovered
color by color by an adversary, while our job is to assign a color immediately after it is recovered. We also prove a more general weighted version of our
result with lists of varying sizes. In consequence we get a simple necessary
and sufficient condition for matroid list colorability in general case. The
main tool we use is the multiple basis exchange property, which we give a simple proof.
\end{abstract}

\maketitle

\section{Introduction}

Let $M$ be a loopless matroid on a ground set $E$. A coloring of the set $E$
is \emph{proper} if elements of the same color form an independent set of $M$%
. The \emph{chromatic number} of $M$, denoted by $\chi (M)$, is the minimum
number of colors needed to color properly the set $E$. In case
of a graphic matroid $M=M(G)$, the number $\chi (M)$ is a well studied
parameter known as the \emph{arboricity} of the underlying graph $G$.

In \cite{se} Seymour considered the following list coloring problem for
matroids, in analogy to the list coloring of graphs. By a simple application
of the matroid union theorem he proved that matroidal version of the choice
number stays the same as the chromatic number.

\begin{theorem}
\emph{(Seymour \cite{se}) }Suppose that every element $e\in E$ of a matroid $%
M$ is assigned a set of colors \thinspace $L(e)$ of size at least $\chi (M)$%
. Then there is a proper coloring $c$ of $M$ satisfying $c(e)\in L(e)$ for
each $e\in E$.
\end{theorem}

In this paper we prove the \emph{on-line} version of Seymour's theorem.
Consider the following game played by Alice and Bob on a matroid $M$, in
analogy to the graph coloring game introduced by Schauz \cite{sc} (cf. \cite%
{zh}). Let $\mathbb{N}=\{1,2,3,\ldots \}$ be the set of colors, and let $k$
be a fixed positive integer. In the first round Bob chooses arbitrary
non-empty subset $B_{1}\subseteq E$ and inserts color $1$ to the lists of
all elements of $B_{1}$. Then Alice chooses some independent set $%
A_{1}\subseteq B_{1}$ and colors its elements by color $1$. In the second
round Bob picks arbitrarily a non-empty subset $B_{2}\subseteq E$ and
inserts color $2$ to the lists of all elements of $B_{2}$. Then Alice
chooses an independent subset $A_{2}\subseteq B_{2}\setminus A_{1}$ and
colors its elements with color $2$. And so on, until all lists will have
exactly $k$ elements. If at the end of the play the whole matroid is
colored, then Alice is the winner. In the opposite case, Bob is the winner.
Let $\widetilde{\chi }(M)$ denote the minimum number $k$ guaranteeing a win
for Alice.

Our main result reads as follows.

\begin{theorem}\label{equal}
Every matroid $M$ satisfies $\widetilde{\chi }(M)=\chi (M)$.
\end{theorem}

The proof relies on a multiple basis exchange property. Actually we prove a
more general result, in which we allow for lists of varying sizes and
coloring by sets of colors. This also gives the fractional version of the theorem 
(for fractional on-line list coloring of graphs see \cite{gu}).

\section{The proof}

We will need some notation. Let $\mathcal{P}(\mathbb{N})$ denote the family
of all subsets of the set of positive integers $\mathbb{N}$ (we use $\mathbb{%
N}$ as the set of colors as well as the set of numbers). Let $\mathbf{w}%
:E\rightarrow \mathbb{N}$ be an assignment of \emph{weights} to the elements
of a matroid $M$. A $\mathbf{w}$\emph{-coloring} of a matroid $M$ is a
function $W:E\rightarrow \mathcal{P}(\mathbb{N})$ such that every coloring $c
$ satisfying $c(e)\in W(e)$ is a proper coloring of $M$. Let $\ell:E\rightarrow 
\mathbb{N}$ be any function and let $L:E\rightarrow \mathcal{P}(\mathbb{N})$
be a list assignment of \emph{size} $\ell$, that is, for each $e\in E$ we have $%
\left\vert L(e)\right\vert =\ell(e)$. We say that $M$ is $\mathbf{w}$\emph{%
-colorable from lists} $L$ if there is a $\mathbf{w}$-coloring $W$ of $M$
satisfying condition $W(e)\subseteq L(e)$ for each $e\in E$.

Now we may consider a generalized game on a matroid $M$ with given functions 
$\mathbf{w}$ and $\ell$, which goes in the same way as described in the
introduction, except that the goal of Alice is a $\mathbf{w}$-coloring of $M$
from lists of size $\ell$. If she has a winning strategy, then we say that $M$
is \emph{on-line }$(\mathbf{w},\ell)$\emph{-colorable}.

Our aim is to prove a sufficient condition for the above property. We need
two simple lemmas. The first is a well-known generalized exchange property.
We will prove this lemma for the sake of completeness.

\begin{lemma}
\label{ex} Let $I_{1}$ and $I_{2}$ be independent sets of a matroid $M$.
Then for every $X\subseteq I_{1}$ there exists $Y\subseteq I_{2}$ such that
both sets, $(I_{1}\setminus X)\cup Y$ and $(I_{2}\setminus Y)\cup X$, are
independent.
\end{lemma}

\begin{proof}
Let $I=I_{1}\cap I_{2}$. We can restrict to the case where $I=\emptyset $.
Indeed, if $I\neq \emptyset $, then consider matroid $M$ with contracted set 
$I$ and two independent sets $I_{1}\setminus I$, and $I_{2}\setminus
I$. For $X\setminus I_{2}$ we get $Y$, which is also good in the
previous case.

Now let $I_{1}\cap I_{2}=\emptyset $. Let $M_{1}$ be matroid $M$ restricted
to the set $X\cup I_{2}$, and let $M_{2}$ be matroid $M$ restricted to the
set $(I_{1}\setminus X)\cup I_{2}$. Let $I_{1}\cup I_{2}$ be their common
ground set, and denote their rank functions by $r_{1},r_{2}$ respectively.
Observe that for each $A\subseteq I_{1}\cup I_{2}$ we have: 
\begin{equation*}
r_{1}(A)+r_{2}(A)=r(A\cap (X\cup I_{2}))+r(A\cap ((I_{1}\setminus X)\cup
I_{2}))\geq 
\end{equation*}%
\begin{equation*}
\geq r(A\cap (I_{1}\cup I_{2}))+r(A\cap I_{2})\geq \lvert A\cap I_{1}\rvert
+\lvert A\cap I_{2}\rvert =\lvert A\rvert ,
\end{equation*}%
where the first inequality is just a submodularity of a rank function. From
the matroid union theorem (see \cite{ox}) it follows that $I_{1}\cup I_{2}$
can be covered by sets $I_{1}^{\prime },I_{2}^{\prime }$ independent in $%
M_{1}$ and $M_{2}$ respectively, so also in $M$. Now $Y=I_{2}\cap
I_{2}^{\prime }$ is a good choice, since $(I_{1}\setminus X)\cup
Y=I_{2}^{\prime }$ and $(I_{2}\setminus Y)\cup X=I_{1}^{\prime }$.
\end{proof}

As a corollary we get the multiple basis exchange property (see \cite{ku,wo}).

\begin{corollary}
\label{exchange}\emph{(Multiple basis exchange property)} Let $B_{1}$ and $%
B_{2}$ be two bases of a matroid $M$. Then for every $X\subseteq B_{1}$
there exists $Y\subseteq B_{2}$, such that $(B_{1}\setminus X)\cup Y$ and $%
(B_{2}\setminus Y)\cup X$ are also bases.
\end{corollary}

We say that a collection of sets $I_{1},\dots ,I_{k}$ is a $\mathbf{w}$\emph{%
-cover} of a set $E$ if for each $e\in E$ we have $\left\vert \{i:e\in
I_{i}\}\right\vert =\mathbf{w}(e)$. For a given subset $U\subseteq E$, let $%
\mathbf{c}_{U}$ denote the characteristic function of $U$, that is, $\mathbf{c}%
_{U}(e)=1$ if $e\in U$ and $\mathbf{c}_{U}(e)=0$, otherwise. Now we prove
the following inductive step lemma.

\begin{lemma}
\label{aaa} Let $I_{1},\ldots ,I_{k}$ be a collection of independent sets in
a matroid $M$ forming a $\mathbf{w}$-cover of its ground set $E$. Then for
every set $V\subseteq E$ there exists an independent set $I\subseteq V$ and
independent sets $I_{1}^{\prime },\dots ,I_{k}^{\prime }$ satisfying the
following conditions.

\begin{enumerate}
\item The sets $I_{1}^{\prime },\dots ,I_{k}^{\prime }$ form a $(\mathbf{w}-%
\mathbf{c}_{I})$-cover of $E$.

\item For each $e\in E$, if $e\in I_{s}^{\prime }$ then $e\in I_{t}$ for
some $t\geq s+\mathbf{c}_{V}(e)$.
\end{enumerate}
\end{lemma}

\begin{proof}
Let $X_{1}=(V\cap I_{1})\setminus (I_{1}\cap I_{2})$. By Lemma \ref{ex}
there exists $Y_{2}\subseteq I_{2}$ such that $I_{1}^{\prime
}=(I_{1}\setminus X_{1})\cup Y_{2}$ and $I_{2}^{\prime \prime
}=(I_{2}\setminus Y_{2})\cup X_{1}$ are independent. In general let $%
X_{i}=(V\cap I_{i}^{\prime \prime })\setminus (I_{i}^{\prime \prime }\cap
I_{i+1})$. So again by Lemma \ref{ex} there exists $Y_{i+1}\subseteq I_{i+1}$%
, such that $I_{i}^{\prime }=(I_{i}^{\prime \prime }\setminus X_{i})\cup
Y_{i+1}$ and $I_{i+1}^{\prime \prime }:=(I_{i+1}\setminus Y_{i+1})\cup X_{i}$
are independent. Let $I=X_{k}$. It is easy to see that conditions (1) and
(2) are satisfied.
\end{proof}

We are ready to prove the following generalization of the theorem of Seymour.

\begin{theorem}
\label{main}Let $\ell$ be a given list-size function on the ground set $E$ of a
matroid $M$. If $M$ is $\mathbf{w}$-colorable from lists of the form $%
L(e)=\{1,2,\dots ,\ell(e)\}$, $e\in E$, then $M$ is on-line $(\mathbf{w},\ell)$%
-colorable.
\end{theorem}

\begin{proof}
We prove it by the induction on the number $\mathbf{w}(E)=\sum_{e\in E}%
\mathbf{w}(e)$. If $\mathbf{w}(E)=0$, then $\mathbf{w}$ is the zero vector
and the assertion holds trivially. Suppose now that $\mathbf{w}(E)\geq 1$
and the assertion of the theorem holds for all $\mathbf{w}^{\prime }$ with $%
\mathbf{w}^{\prime }(E)<\mathbf{w}(E)$. Let $V\subseteq E$ be the set of
elements picked by Bob in the first round of the game. So, all elements of $V
$ have color $1$ in their lists. Let $I_{1},\dots ,I_{k}$ be a $\mathbf{w}$%
-coloring of $M$ which exists by the assumption. By Lemma \ref{aaa}, there
exist independent sets $I\subseteq V$ and $I_{1}^{\prime },\dots
,I_{k}^{\prime }$, such that $I_{1}^{\prime },\dots ,I_{k}^{\prime }$ is a $(%
\mathbf{w}-\mathbf{c}_{I})$-cover of $E$. Now Alice colors all elements from 
$I$ with color $1$. By condition (2) of Lemma \ref{aaa}, matroid $M$ is $(%
\mathbf{w}-\mathbf{c}_{I})$-colorable from lists $L^{\prime }(e)=\{1,2,\dots
,\ell(e)-\mathbf{c}_{V}(e)\}$. The assertion of the theorem follows by
induction.
\end{proof}

Observe that the condition from the assumption of Theorem \ref{main} is not
only sufficient, but also a necessary for a matroid to be on-line $(\mathbf{w}%
,\ell)$-colorable. 

Taking $\mathbf{w}=(1,1,\ldots ,1)$ and $l=(k,k,\ldots ,k)$,
with $k=\chi (M)$, we get immediately Theorem \ref{equal}. Theorem \ref{main} 
is an on-line generalization of Theorem 3 form \cite{la}. Let us mention one of 
this off-line consequences.

\begin{corollary}
If $M$ is colorable from lists of the form $L(e)=\{1,2,\dots ,\ell(e)\}$, $e\in
E$, then $M$ is colorable from any lists of size $\ell$.
\end{corollary}

\section*{Acknowledgements}

We would like to thank Jarek Grytczuk for many inspiring conversations, and
additionally for the help in preparation of this manuscript.

\end{document}